\documentclass[11pt,leqno]{amsart}
\usepackage[letterpaper, total={6in, 8in}, left=1.25in, top=1.5in]{geometry}
\usepackage{amsmath}
\usepackage{amsfonts}
\usepackage{amssymb}
\usepackage{graphicx}
\usepackage{color}
\usepackage{hyperref}
\usepackage{xcolor}

\newtheorem{theorem}{Theorem}[section]

\newtheorem{lemma}{Lemma}[section]

\newtheorem{remark}[theorem]{Remark}

\def\k{\kappa}

\def\l{\lambda}

\def\p{\partial}

\def\S{\Bbb S}
\def\R{\mathbb{R}}

\def\Sph{\Bbb S}

\def\k{\kappa}

\def\H{\mathbb{H}}

\def \p {\partial}

\def\V{\operatorname{Vol}}

\def\tr{\operatorname{tr}}

\def\tr{\operatorname{tr}}

\numberwithin{equation}{section}

\begin{document}

\title[Neumann eigenvalues of the Witten-Laplacian]{An isoperimetric inequality for  Neumann eigenvalues with radial log-concave measures}

\author{Sun xiaomei } 
\address{College of informatics, Huazhong Agricultural University, 430070, Wuhan, China}
\email{xmsunn@mail.hzau.edu.cn}

\author{Kui Wang} 
\address{School of Mathematical Sciences, Soochow University, Suzhou, 215006, China}
\email{kuiwang@suda.edu.cn}

\author{Anqiang Zhu}
\address{School of Mathematics and Statistics, Wuhan University, Wuhan 430072, China}
\thanks{Corresponding author. Email: \href{mailto:aqzhu.math@whu.edu.cn}{aqzhu.math@whu.edu.cn} (A. Zhu).}
\email{aqzhu.math@whu.edu.cn}

\begin{abstract}
 We prove a sharp isoperimetric inequality for the harmonic mean of the first $n$ nonzero Neumann eigenvalues of the Witten-Laplacian on origin-symmetric Lipschitz domains in space forms, endowed with radial log-concave measures. The main novelty is that we establish the sharp harmonic mean inequality under general radial log-concave measures, without assuming the weight function  to be non-increasing. This extends previous results that were restricted to specific or more restrictive weighted settings. The proof relies on a refined analysis of the first eigenfunction on geodesic balls, a monotonicity property derived from a convexity condition on the radial weight, and a matrix trace inequality. The result recovers, as special cases, the Gaussian‑space inequalities of Chiacchio and Di Blasio \cite{CB12}, as well as those of Chiacchio \cite{Fr26} and Gao-Wang \cite{GW26}, while also improving recent work of Chen-Mao \cite{CM24}.

\end{abstract}



\maketitle

\section{Introduction}
The classical Szeg\"o-Weinberger inequality \cite{Sz54, Wei56} asserts that, among bounded domains of fixed volume in Euclidean space, the ball uniquely maximizes the first nonzero Neumann eigenvalue. This result has been extended to bounded domains in the hemisphere and in hyperbolic space \cite{AB95}, as summarized in Ashbaugh and Benguria’s foundational survey of open spectral conjectures \cite{As99}. For further  developments, we refer to \cite{AS96, BCB16,  LL23, Wang19} and references therein. A  related open conjecture posed by Ashbaugh and Benguria \cite{AB93} concerns the sum of reciprocals of low-order Neumann eigenvalues: for any bounded Lipschitz domain $\Omega\subset\mathbb{R}^n$,
\begin{align}\label{eq1.1}
\sum_{i=1}^n \frac{1}{\mu_i(\Omega)} \ge \frac{n}{\mu_1(B)},
\end{align}
where $B$ denotes a ball with the same volume as $\Omega$.
Xia and Wang \cite{XW23}  proved the sharp lower bound for the harmonic mean of the first $n-1$ nonzero Neumann eigenvalues in Euclidean and hyperbolic spaces; analogous results were derived for hemispheres \cite{BBC20}, rank-one symmetric spaces \cite{MW24}, and Witten-Laplacians associated with weighted volume forms \cite{CM24}. All these works confirm that the ball is the unique maximizer of $\mu_1$ and simultaneously minimizes the sum $\sum_{i=1}^{n-1}1/\mu_i$. More recently, He, Li and Tang \cite{HLT26} proved \eqref{eq1.1} for Euclidean space. See also \cite{YZ26} for space forms. 

Weighted spectral problems under Bakry-\'Emery measures (also called $\phi$-volume measures) provide a natural generalization of the classical Laplacian framework and include the Gaussian space as an  important example. The Gaussian measure $(2\pi)^{-m/2}e^{-|x|^2/2}dx$ corresponds to a convex weight potential $\phi$, and the associated Hermite operator $-\Delta + x\cdot\nabla$ replaces the standard Laplacian with a drift term induced by the Gaussian weight.  In the Gaussian setting,  Chiacchio and Di Blasio \cite{CB12} established a Szeg\"o-Weinberger type inequality: among origin-symmetric Lipschitz domains with the same Gaussian volume, the centered Euclidean ball uniquely maximizes the first nonzero Neumann eigenvalue. Building on their work, Gao and the second author \cite{GW26} extended this single-eigenvalue estimate to a sharp isoperimetric inequality governing the harmonic mean of the first $n-1$ nonzero Neumann eigenvalues in Gauss space. More recently, Chiacchio  \cite{Fr26} proved the corresponding result for the first $n$ nonzero Neumann eigenvalues in the Gaussian setting.

Despite the rich literature on Euclidean, constant-curvature, and Gaussian weighted spaces, a unified geometric framework covering general radial weighted manifolds $M^n$ remained absent. Recently,  Chen and Mao \cite{CM24} extended the result of Xia-Wang \cite{XW23} to space forms with log concave measure $(M,e^{-\phi(r(x))}d\V)$. However, their approach requires the weight function  $\phi(r)$ to be both nonincreasing and convex.  The present paper aims to remove the non-increasing conditions on $\phi(r)$ and to establish sharp isoperimetric inequalities for Neumann eigenvalues under general convex weight potentials $\phi(r)$. 

To state the main theorem, we introduce the following notation.  Let $M_\k$ denote the  $n$-dimensional complete simply connected Riemannian manifold  with constant sectional curvature $\kappa\in \{1,0,-1\}$, i.e., $M_1=\S^n$, $M_{0}=\R^n$ and $M_{-1}=\H^n$. Let $o$ be the origin of $M_\k$, and let $B_R$ denote the geodesic ball of radius $R$ centered at  $o$. Consider the radial weighted measure
$$
d\gamma_\k=e^{-\phi(r(x))}\, d\V_\k,
$$
where  $d\V_\k$  is the Riemannian volume element of $M_\k$  and $\phi$ is a smooth radial function centered at  $o$ and is smooth on $M_\k$. For a domain $\Omega\subset M_\k$, its volume is given by 
$$
|\Omega|_{\gamma_\k}=\int_\Omega d\gamma_\k=\int_\Omega e^{-\phi(r(x))}\, d\V_\k.
$$
Let $\mathcal{E}_\k$ denote the class of connected Lipschitz domains $\Omega\subset M_\k$  that are symmetric about $o$, such that $|\Omega|_{\gamma_\k}< |M_\k|_{\gamma_\k}$ when $\k=0$ or $\k=-1$, and such that $\Omega$ is contained in the $o$-centered geodesic ball $B_{\pi/2}$ when $\k=1$. Define the sine-type function $S_\k(r)$ as
\begin{align*}
    S_\kappa(r) =
\begin{cases}
\sin r, & \text{if } \kappa=1,\\
r, & \text{if } \kappa=0,\\
\sinh r, & \text{if } \kappa=-1,
\end{cases}\text{ and  $C_\kappa(r) = S_\kappa'(r)$.}
\end{align*}
We study the Neumann eigenvalue problem for the Witten-Laplacian  $\Delta_\phi:=\Delta-\nabla \phi \cdot \nabla$ on  $\Omega\in \mathcal{E}_\k$:
\begin{align}\label{1.2}
    \begin{cases}
        -\Delta u+\nabla u \cdot \nabla \phi=\mu u, &\quad \text{in $\Omega$},\\
        \dfrac{\p u}{\p \nu}=0, &\quad \text{on $\p \Omega$},
    \end{cases}
\end{align}
where $\p u/\p \nu$ denotes the outward normal derivative. The spectrum consists of discrete eigenvalues, denoted by $\mu_k(\Omega)$ for $k=0,1,\cdots$, satisfying
  \begin{align*}
       0=\mu_0(\Omega) < \mu_1(\Omega) \leq \mu_2(\Omega) \leq \cdots \to +\infty,
    \end{align*}
with each eigenvalue repeated according to multiplicity.

Our main theorem advances the theory of Ashbaugh–Benguria-type reciprocal sum inequalities to the general weighted manifold setting, under the additional monotonicity condition: $$\big(\frac{C_{\k}(r)}{S_{\k}(r)}\phi'(r)\big)'\ge0.$$
Precisely, we establish the following result.
\begin{theorem}\label{thm1.1}
With the notation above, let  $\Omega\in \mathcal{E}_\k$, and  let $B_R\subset M_\k$ be the geodesic ball of radius $R$ centered at  $o$ such that $|\Omega|_{\gamma_\k}=|B_R|_{\gamma_\k}$. Suppose $\phi(r)$ is convex and satisfies 
$$ \big(\frac{C_{\k}(r)}{S_{\k}(r)}\phi'(r)\big)'\geq 0, \quad 0<r<R.$$
 Then 
 \begin{align}\label{1.3}
\frac{1}{\mu_{1}(\Omega)} + \frac{1}{\mu_{2}(\Omega)} + \dots + \frac{1}{\mu_{n}(\Omega)} \ge \frac{n}{\mu_{1}\big(B_R\big)}.
\end{align}
Moreover, equality holds if and only if $\Omega$ coincides with the ball $B_R$. Consequently,
\begin{align}\label{1.4}
    \mu_1(\Omega)\le \mu_1(B_R).
\end{align}
\end{theorem}

\begin{remark}
    Since $\mu_1(\Omega)\le\mu_2(\Omega)\le\cdots\le \mu_n(\Omega)$, inequality \eqref{1.3} directly implies
    \begin{align*}
    \frac{n}{\mu_{1}\big(B_R\big)}\le    \sum_{i=1}^n \frac{1}{\mu_{i}(\Omega)}\le \frac{n}{\mu_{1}(\Omega)}, 
    \end{align*}
  from  which \eqref{1.4} holds. Thus we obtain a new Szeg\"o-Weinberger inequality, which in particular recovers the Gaussian weighted space result of Chiacchio and Di Blasio \cite{CB12} as a special case.
\end{remark}

Theorem \ref{thm1.1} extends the recent Gaussian space results of  Chiacchio \cite{Fr26} and Gao-Wang \cite{GW26}, and improves the results of Chen-Mao \cite{CM24}. In \cite{CM24}, the function $\phi(r)$ is assumed to be a non-increasing convex function, a condition not satisfied by the Gaussian measure. In fact, the conditions in Theorem \ref{thm1.1} are  met by a wide range of examples; for instance, $\R^{n}$ with $\phi(r)=r^{2k}$ for $k\ge 1$. 

The novelty of our approach lies in both the scope of the result and the methodology.  The main technical difficulty, however, resides in the weighted setting: one must establish, under general log-concave measures, the precise structure and monotonicity properties of the first eigenfunctions on geodesic balls.  We adopt a different strategy from that of Chiacchio and Di Blasio \cite{CB12}  to characterize the first nonzero eigenfunction of the Witten-Laplacian on balls. This approach, presented in Section \ref{sect3}, is relatively straightforward and, more importantly, applies to a broad class of convex radial weights, without requiring the monotonicity of $\phi$ as in \cite{CM24}
  or the special structure of the Gaussian density.  Consequently, our theorem yields a sharp Szeg\"o-Weinberger   inequality for general radial log-concave measures on space forms, thereby extending the Gaussian space results in \cite{Fr26, GW26} and providing a unified treatment that was not previously available in the literature.

The structure of this paper is organized as follows. Section \ref{sect2} collects the necessary preliminaries, including the weighted manifold setting, the Neumann eigenvalue problem for the Witten-Laplacian, and several auxiliary lemmas.  Section \ref{sect3}
is devoted to the study of the Neumann eigenvalue problem on geodesic balls, including the precise form and multiplicity of the first nonzero eigenfunctions, and  key monotonicity properties. In Section  \ref{sect4}, we prove Theorem \ref{thm1.1}.

\section{Preliminaries}\label{sect2}
We begin by recalling some standard facts about geodesic polar coordinates on the space form $M_\k$. In geodesic polar coordinates $(r, \theta)$ centered at the origin $o$, the  Riemannian metric is $$g_\k=dr^{2}+S_{\kappa}^{2}(r)g_{\S^{n-1}},$$
where $g_{\S^{n-1}}$ denotes the standard metric on the unit sphere $\S^{n-1}$ and $r(x)$ is the geodesic distance from $o$ to $x$. The  corresponding volume element is given by
$$d\V_\k=S_{\kappa}(r)^{n-1}drd\theta,$$
with $d\theta$ denotes the volume element on $\S^{n-1}$.
The Laplace-Beltrami operator $\Delta$  takes the explicit form
\begin{align*}
    \Delta=\frac{1}{S_\k(r)^{n-1}}\frac{\partial }{\partial r}\left(S_\k(r)^{n-1}\frac{\p}{\p r}\right)+\frac{1}{S_\k(r)^2}\Delta_{\S^{n-1}}
\end{align*}
For a smooth radial function $\phi(r(x))$ defined on $M_{\kappa}$, 
the Witten-Laplacian $ \Delta_{\phi}$ is defined by
\begin{align*}
    \Delta_{\phi}=\Delta-\nabla\phi\cdot \nabla=\frac{1}{S_\k(r)^{n-1}}\frac{\partial}{\partial r}\left(S_\k(r)^{n-1}\frac{\p}{\p r}\right)+\frac{1}{S_\k(r)^2}\Delta_{\S^{n-1}}-\phi'(r)\frac{\p}{\p r}.
\end{align*}
Throughout this paper, $\Omega \subset M_\k$ is assumed to be a  connected Lipschitz domian that is symmetric about the origin,  such that $|\Omega|_{\gamma_\k}< |M_\k|_{\gamma_\k}$ when $\k=0$ or $\k=-1$, and such that $\Omega$ is contained in $B_{\pi/2}$ when $\k=1$. Furthermore, we assume that the compact embedding
\begin{equation}\label{eq:compact-embedding}
H^1(\Omega,\gamma_{\k})\hookrightarrow L^2(\Omega,\gamma_{\k})
\quad\text{compactly}
\end{equation}
holds; see, for instance,  \cite[Definition 2.1 and Remark 2.1]{CG22} for  further details.

 For a matrix $A$, we denote by $\tr(A)$ its trace, and we let $I$ denote the  $n\times n$ identity matrix throughout the paper.
 For real symmetric matrices $A$ and $B$, we write $A\succeq B$ (respectively, $A\preceq B$)   to indicate that $A-B$ is  positive semidefinite (respectively, negative semidefinite), and $A\succ B$ (respectively, $A\prec B$)   to indicate that $A-B$ is  positive definite (respectively, negative definite).
We first recall  a matrix inequality from \cite{HLT26}, which will serve as a useful tool in the proof of Theorem \ref{thm1.1}. 

\begin{lemma}[\cite{HLT26}]\label{lm2.1}
Let $a,c,d,\lambda>0$, and let $Z$ be a real symmetric $n\times n$
matrix with $\operatorname{tr} Z=0$. Suppose  $J$ and $K$ are real symmetric
matrices such that
\begin{align*}
        J\succeq0,
        \qquad
        J\succeq aI+cZ,
        \qquad
        0\prec K\preceq \lambda aI-dZ.
\end{align*}
Then
\begin{equation}\label{eq:matrix-goal}
        \tr(K^{-1}J)\geq \frac{n}{\lambda}.
\end{equation}
Moreover, if equality holds in \eqref{eq:matrix-goal}, then $Z=0$,
$J=aI$, and $K=\lambda aI$. 
\end{lemma}

Next, we state a standard trace formulation of Hersch’s variational principle for reciprocal sums of Neumann eigenvalues. We refer to \cite{HerschCRAS} for the original result, and refer to \cite{Fr26, Hile&Xu}  for related extensions and discussions.

\begin{lemma}\label{lm2.2}
Let $\Omega\in \mathcal{E}_\k$. Denote $\mu_m(\Omega)$ be the $m$-th Neumann eigenvalue of \eqref{1.2}.
Suppose
$P_1, \ldots, P_n\in H^1(\Omega,\gamma_{\kappa})$ are linearly independent in
$L^2(\Omega,\gamma_{\k})$ satisfying
\begin{align*}
        \int_\Omega P_i\,d\gamma_{\kappa}=0,
        \qquad i=1,\ldots,n.
\end{align*}
Define the matrices $J=(J_{ij})$ and $K=(K_{ij})$ by
\begin{align*}
            J_{ij}=\int_\Omega P_iP_j\,d\gamma_{\kappa},
        \qquad
        K_{ij}=\int_\Omega \langle \nabla P_i, \nabla P_j\rangle_{M_\k}\,d\gamma_{\kappa}.
\end{align*}
If $K$ is positive definite, then
\begin{align*}
          \sum_{i=1}^n \frac1{\mu_i(\Omega)}\geq \tr(K^{-1}J).
\end{align*}
Where   $\langle \cdot, \cdot\rangle_{M_{\k}}$  denotes the inner product with respect to the metric of $M_\k$.
\end{lemma}

Finally, we recall a weighted one-dimensional version of the bathtub principle, which will be used in the proof of the main theorem. For the classical formulation and related results, we refer the reader to \cite[Theorem~1.14]{LiebLoss}, \cite[Lemma~3.2]{HLT26} or \cite[~Lemma 4.1]{Fr26}.

\begin{lemma}\label{lm2.3}
Let $0<L\leq\infty$, and let $\rho\geq0$ be locally integrable on $[0,L)$, and assume $\rho>0$ almost everywhere in $(0,L)$. Set
\begin{align*}
d\mu=\rho(r)\,dr,
\qquad
V(t)=\mu((0,t)),
\qquad
Y_*=\mu((0,L)).
\end{align*}
Then $V$ is continuous and strictly increasing on $[0,L)$. For $0\leq y<Y_*$, let $r_y=V^{-1}(y)$. If $Y_*<\infty$, we also allow $y=Y_*$ and set $r_{Y_*}=L$, with the convention that $(0,L)=(0,\infty)$ when $L=\infty$. Let $E\subset(0,L)$ be measurable with $\mu(E)=y$. Then the following hold:
\begin{enumerate}
\item If $w$ is nondecreasing on $[0,L)$, then
\begin{align*}
\int_E w\,d\mu\geq\int_0^{r_y}w\,d\mu;
\end{align*}
\item If $w$ is nonincreasing on $[0,L)$, then
\begin{align*}
\int_E w\,d\mu\leq\int_0^{r_y}w\,d\mu.
\end{align*}
\end{enumerate}
In particular, if $0<y<Y_*$ and $w$ is strictly decreasing, equality in the second inequality holds if and only if $E=(0,r_y)$ up to a $\mu$-null set.
\end{lemma}

\section{Neumann eigenvalue problem on geodesic balls}\label{sect3}
In this section, we establish several properties of the first nonzero Neumann eigenvalue and its associated eigenfunctions on geodesic balls. Let $M_\k$  be the $n$ dimensional  complete simply connected Riemannian manifold of constant sectional curvature $\kappa\in\{-1,0,1\}$, and let  $B_R$ denote the origin-centered ball with radius $R$ in $M_\k$. The Neumann eigenvalue problem of \eqref{1.2} on $B_R$ reads
\begin{align}\label{3.1}
   \begin{cases}
       -\Delta u+\nabla u \cdot \nabla \phi=\mu u, \quad& \text{in $B_R$,}\\
       \frac{\p u}{\p \nu}=0, \quad& \text{on $\p B_R$,}
   \end{cases} 
\end{align}
where $\nu$ denotes the unit outward normal on  $\p B_R$.

\begin{lemma}\label{lm3.1}
Let $B_R\subset M_\k^n$ be the origin-centered ball of radius $R$ with $R<\pi$ if $\k=1$. Suppose $\phi$ is  radial and convex. Then the first nonzero eigenvalue  of \eqref{3.1} has muliplicity $n$, and all the corresponding eigenfunctions are of the form
$$
u(x)=T(r)\psi_i(\theta), \quad i=1,2, \cdots, n,
$$
where $(r, \theta)$ are polar coordinates, $\psi_i(\theta)$ are the linear coordinate functions restricted to $\mathbb{S}^{n-1}$, and 
$T(r)$ is the first eigenfunction of the boundary value problem
\begin{align}\label{3.2}
    \begin{cases}
T'' + \left((n-1)\frac{C_\kappa}{S_\kappa} - \phi'\right)T' + \big(\lambda - (n-1)S_\kappa^{-2}\big)T = 0,\quad r\in (0, R),\\
T(0) = 0,\quad T'(R) = 0. 
\end{cases}
\end{align}
\end{lemma}
\begin{proof}
Since the  metric on $B_R$ admits  the polar decomposition $$g_\k=dr^2+S_\k(r)^2 g_{\S^{n-1}},$$  and $\phi$ is radial symmetric, we may apply the separation of variables
technique to find solutions of \eqref{3.1}. 
Let $u(x)=T(r)\psi(\theta)$ be an eigenfunction of \eqref{3.1}. A   direct computation yields
\begin{align}\label{3.3}
   -\frac{T'' + \left[(n-1)\frac{C_\kappa}{S_\kappa} - \phi'\right]T'}{T} -\frac 1 {S_\k^2}\frac{\Delta_{\S^{n-1}}\psi}{\psi}=\mu.
\end{align}
where $\Delta_{\S^{n-1}}$ denotes the  Laplace-Beltrami operator  on $\S^{n-1}$. Thus $\psi$ is an eigenfunction of $\Delta_{\S^{n-1}}$.
Recall that  the eigenvalues of $\Delta_{\S^{n-1}}$ are $k(k+n-2)$  for $k=0, 1, 2, \cdots$. Consequently, from  \eqref{3.3} we obtain 
\begin{align*} 
T'' + \left[(n-1)\frac{C_\kappa}{S_\kappa} - \phi'\right]T' + \left(\mu - \frac{k(k+n-2)}{S_\kappa^2}\right)T = 0.
\end{align*}
It follows that the first nonzero eigenvalue $\mu_1(B_R)$ of  \eqref{3.1}  is either the second eigenvalue $\tau_2$ of
 \begin{align}
     \begin{cases}\label{3.4}
T'' + \left[(n-1)\frac{C_\kappa}{S_\kappa} - \phi'\right]T'+ \tau T = 0, & \text{in } (0,R),\\
T'(0) = 0,\quad T'(R)  = 0,
\end{cases}
 \end{align}
or the first eigenvalue $\l_1$ of 
\begin{align}\label{3.5}
    \begin{cases}
T'' + \left[(n-1)\frac{C_\kappa}{S_\kappa} - \phi'\right]T'+ (\l-\frac{n-1}{S_\k^2})T=0, & \text{in } (0,R),\\
T(0) = 0,\quad T'(R) = 0.
\end{cases}
\end{align}
We now show $\l_1<\tau_2$.

 Let $g$ and $f$ be eigenfunctions associated to $\l_1$ and $\tau_2$, respectively.
By the Sturm-Liouville theorem,  $g$ has constant sign on on $(0,R)$ and we assume without loss of generality that $g(r)>0$ for $r\in (0,R)$. The function $f(r)$ has two nodal sets: there exists $r_{0}\in (0,R)$ such that $f(r)>0,r\in (0,r_{0})$ and $f(r)<0,r\in (r_{0},R)$.
Set
$$
p(r)=S_{\k}(r)^{n-1}e^{-\phi(r)}.
$$
Then  \eqref{3.4} can be rewritten as 
\begin{align*}
\begin{cases}
\left(p(r)f'(r)\right)' + \tau_2 \,p(r)\,f(r) = 0, \quad r\in (0, R),\\
f'(0) = f'(R) = 0,
\end{cases}
\end{align*}
Integrating over $(0,R)$ gives
\begin{align}\label{3.6}
    \int_{0}^{R}\tau p(r)f(r)dr=-p(r)f'(r)\Big|_{0}^{R}=0.
\end{align}
Since $f(r)<0$ for $r\in (r_{0},R)$, we deduce from \eqref{3.6} that
\begin{align*}
    0<\int_{0}^{r}\tau_2 p(r)f(r)dr=-p(r)f'(r), \quad r\in (0, R).
\end{align*}
 Hence $f'(r)<0 $ on $ r\in (0,R)$.
Let $h(r)=f'(r)<0$. Differentiating \eqref{3.4}, we derive
\begin{align}\label{3.7}
   h'' + h'\left(\dfrac{(n-1)C_{\kappa}}{S_{\kappa}} - \phi'\right) + \tau_2 h+(-\frac{n-1}{S_{\k}^{2}}-\phi'')h = 0, \quad \text{in } (0,R).
\end{align}
Since $\phi(r)$ is convex, we have $-\phi''(r)h(r)\geq 0$.  Thus  \eqref{3.7} implies the differential inequality
\begin{align}\label{3.8}
    h'' + h'\left(\dfrac{(n-1)C_{\kappa}}{S_{\kappa}} - \phi'\right) + \tau_2 h-\frac{n-1}{S_{\k}^{2}}h \leq  0, \quad \text{in } (0,R),
\end{align}
with boundary conditions  $h(0)=0$ and $h(R)=0$.

Comparing \eqref{3.8} with  \eqref{3.5}, we obtain 
\begin{align}\label{3.9}
    &\int_{0}^{R}\left(h''(r)g(r)-h(r)g''(r)\right)p(r)dr\nonumber\\
    +&\int_{0}^{R}\left(\dfrac{(n-1)C_{\kappa}(r)}{S_{\kappa}(r)} - \phi'(r)\right) \left(h'(r)g(r)-h(r)g'(r)\right)p(r)dr\nonumber\\
    +&\int_{0}^{R}(\tau_2-\l_1)h(r) g(r)p(r)dr\nonumber\\
    \leq& 0.
\end{align}
Integration by parts gives
\begin{align}\label{3.10}
    &\int_{0}^{R}(h''(r)g(r)-h(r)g''(r))p(r)dr\nonumber\\
    =&\int_{0}^{R}(h'(r)g(r)-h(r)g'(r))'p(r)dr\nonumber\\
    =&\int_{0}^{R}((h'(r)g(r)-h(r)g'(r))p(r))'dr-\int_{0}^{R}(h'(r)g(r)-h(r)g'(r))p'(r)dr.
\end{align}
Substituting  \eqref{3.10} into \eqref{3.9}, we have 
\begin{align} \label{3.11}
    &(\tau_2-\l_1)\int_{0}^{R}h(r)g(r)p(r)dr\nonumber\\
    \leq& -(h'(r)g(r)-h(r)g'(r))p(r)\Big|_{0}^{R}\nonumber\\
    =&-(h'(R)g(R)-h(R)g'(R))p(R)\nonumber\\
    =&-h'(R)g(R)p(R).
\end{align}
Since $h(r)\leq 0$ and $h(R)=0$, we have $h'(R)>0$.  Hence  $-h'(R)g(R)p(R)< 0$.
Because $g(r)> 0$ and $h(r)< 0$ on $r\in (0,R)$, it follows from \eqref{3.11} that
\begin{align*}
    \l_1< \tau_2.
\end{align*}
Thus $\mu_1(B_R)=\l_1$. Recalling that the first nonzero eigenvalue of $\Delta_{\S^{n-1}}$ has muliplicity $n$ with corresponding eigenfucntions  $\psi_i(\theta)$ given by  the linear coordinate functions restricted to $\S^{n-1}$, we conclude that the first nonzero eigenvalue of \eqref{3.1} has muliplicity $n$ and the corresponding eigenfucntions are given by $u_i(r, \theta)=T_1(r)\psi_i(\theta)$, where $T_1(r)$ is the first eigenfunction of \eqref{3.2}.
\end{proof}

\begin{lemma}\label{lm3.2}
  The function  $T_{1}(r)$ is strictly increasing on $(0,R)$, where $R\in (0,\infty)$ for $\kappa=0,-1$ and $R\in (0,\frac{\pi}{2})$ for $\kappa=1$.
\end{lemma}
\begin{proof}
   Equation \eqref{3.2} can be rewritten as 
   \begin{align}\label{3.12}
       (e^{-\phi(r)}S_{\kappa}^{n-1}(r)T'_{1}(r))'=-(\mu_1(B_R)-\frac{n-1}{S_{\k}^{2}(r)})T_{1}(r)e^{-\phi(r)}S_{\k}^{n-1}(r).
   \end{align}
   Integrating  \eqref{3.12} over $(0,R)$ and using the boundary condition $T_1'(R)=0$, we obtain
   \begin{align*}
        &\int_{0}^{R}(\mu_1(B_R)-\frac{n-1}{S_{\k}^{2}(r)})T_{1}(r)e^{-\phi(r)}S_{\k}^{n-1}(r)dr\\
        =& -(e^{-\phi(r)}S_{\kappa}^{n-1}(r)T'_{1}(r))\Big|_{0}^{R}\\
        =&-e^{-\phi(R)}S_{\kappa}^{n-1}(R)T'_{1}(R)\\
        =&0.
   \end{align*}
  For $\kappa=0,-1$, the function $\mu_1(B_R)-\frac{n-1}{S_{\k}^{2}(r)}$ is strictly increasing on $(0,R)$ for any $R>0$. For $\kappa=1$, this function is strictly increasing on $(0,R)$ provided $R\in (0,\frac{\pi}{2})$. Hence, in all cases under consideration, for every $r\in (0,R)$, we have 
  \begin{align*}
       0&> \int_{0}^{r}(\mu-\frac{n-1}{S_{\k}(r)^{2}})T_{1}(r)e^{-\phi(r)}S_{\k}(r)^{n-1}dr= -(e^{-\phi(r)}S_{\kappa}(r)^{n-1}T'_{1}(r))\Big|_{0}^{r}\nonumber\\
       &=-e^{-\phi(r)}S_{\kappa}(r)^{n-1}T'_{1}(r).
  \end{align*}
Thus $T_1'(r)>0$ on $(0, R)$, so $T_{1}$ is strictly increasing.
\end{proof}

\begin{lemma}\label{lm3.3}
   If $(\frac{C_{\k}(r)}{S_{\k}(r)}\phi'(r))'\geq 0$, the function $\frac{T_{1}(r)}{S_{\k}(r)}$ is monotonically decreasing on $(0,R)$.
\end{lemma}
\begin{proof}
   It suffices to prove that  $(\frac{T_{1}(r)}{S_{\k}(r)})'\leq 0$, which is equivalent to 
\begin{align}\label{3.13}
\frac{T'_{1}(r)}{T_{1}(r)}\leq \frac{C_{\k}(r)}{S_{\k}(r)}, \quad r\in (0, R).   
\end{align}
    Let $v(r)=\frac{T_{1}'(r)}{T_{1}(r)}$. From  \eqref{3.2}, a direct computation gives 
    \begin{align*}
        v'(r)&=\frac{T_{1}''(r)}{T_{1}(r)}-(\frac{T'_{1}(r)}{T_{1}(r)})^{2}\nonumber\\
        &=-\frac{\left(\dfrac{(n-1)C_\kappa}{S_\kappa} - \phi'\right)T_{1}'(r) + \big(\mu_{1}(B_R) - (n-1)S_\kappa^{-2}\big)T_{1}(r)}{T_{1}(r)}-v(r)^{2}\nonumber\\
        &=-\left(\dfrac{(n-1)C_\kappa}{S_\kappa} - \phi'\right)v(r)-v(r)^{2}-\big(\mu_{1}(B_R) - (n-1)S_\kappa^{-2}\big),
    \end{align*}
and  consequently,
    \begin{align}\label{3.14}
        v''&=-\left(\dfrac{(n-1)C_\kappa}{S_\kappa} - \phi'\right)v'+(\phi''+\frac{n-1}{S_{\k}^{2}})v-2vv'-\frac{2(n-1)}{S_{\k}^{3}}C_{\k}.
    \end{align}
Now set $$w(r)=v(r)-\frac{C_{\k}(r)}{S_{\k}(r)}.$$
It follows from \eqref{3.14} that
\begin{align*}
    (w(r)+\frac{C_{\k}(r)}{S_{\k}(r)})''=&-\left((n-1) \frac{C_\k}{S_\k} - \phi'\right)(w+\frac{C_\k}{S_\k} )'+(\phi''+\frac{n-1}{S_{\k}^{2}})(w+\frac{C_{\k}}{S_{\k}})\\
    &-2(w+\frac{C_\k}{S_\k} )(w+\frac{C_\k}{S_\k} )'-\frac{2(n-1)}{S_{\k}^{3}}C_{\k}\\
    =&-\left((n-1) \frac{C_\k}{S_\k}  - \phi'+2v\right) w'+(\phi''+\frac{n+1}{S_{\k}^{2}})w\\
    &+\frac{C_{\k}}{S_{\k}}\phi''+(\frac{C_{\k}}{S_{\k}})'\phi'+\frac{2C_\k}{S_\k^3}, 
\end{align*}
and  further simplification yields
\begin{align}\label{3.15}
    w''=-\left(\frac{(n-1)C_\kappa}{S_\kappa} - \phi'+2v\right) w'+(\phi''+\frac{n+1}{S_{\k}^{2}})w
    +\left(\frac{C_{\k}}{S_{\k}}\phi'\right)'.
\end{align}
where we  have used the identity $(\frac{C_\k}{S_\k})'=-\frac 1 {S_\k^2}$. 

Suppose, for contradiction, that  $w$ attains a positive maximum at some 
 $r_{0}\in (0,R)$. Then $w'(r_{0})=0$ and $w''(r_{0})\leq 0$. Using  \eqref{3.15} and the assumption $(C_\k/S_\k \phi')'\ge 0$, we estimate
\begin{align*}
     0 \geq& w''(r_{0})\geq (\phi''(r_{0})+\frac{n+1}{S_{\k}(r_{0})^{2}})w(r_{0})+\left(\frac{C_{\k}}{S_{\k}}\phi'\right)'(r_0)\\
     \ge & (\phi''(r_{0})+\frac{n+1}{S_{\k}(r_{0})^{2}})w(r_{0})\\
     >&0
\end{align*}
which is a contradiction. Therefore $w$ cannot have a positive interior maximum.
 We  now examine the boundary behavior of 
$w(r)$. At $r=R$, we have
 $$w(R)=-\frac{C_{\k}(R)}{S_{\k}(R)}<0.$$ 
 For $r\rightarrow 0^{+}$, using $T_1(0)=0$ and $T_1''(0)=0$ (since $T_1(r)\psi_i$ is smooth in $B_R$), we compute 
\begin{align*}
    \lim_{r\rightarrow 0^{+}}w(r)&=\lim_{r\rightarrow 0^{+}}\frac{S_{\k}(r)T_{1}'(r)-T_{1}(r)S_{\k}'(r)}{T_{1}(r)S_{\k}(r)}\\
    &=\lim_{r\rightarrow 0^{+}}\frac{S_{\k}(r)T_{1}''(r)-T_{1}(r)S_{\k}''(r)}{T_{1}'(r)S_{\k}(r)+T_{1}(r)S_{\k}'(r)}\\
    &=\lim_{r\rightarrow 0^{+}}\frac{T_{1}''(r)+\k T_{1}(r)}{T_{1}'(r)+S_{\k}'(r)\frac{T_{1}(r)}{S_{\k}(r)}}\\
    &=0.
\end{align*}
Thus $w(r)\leq 0$ on $(0,R)$, which proves \eqref{3.13}.
\end{proof}
\begin{remark}
    In the case $\k=0$ and $\phi(r)=r^{2}/2$, we have $\frac{C_{\k}(r)}{S_{\k}(r)}\phi'(r)=1$, and hence $(\frac{C_{\k}(r)}{S_{\k}(r)}\phi'(r))'=0$. Therefore Lemma  \ref{lm3.3}  applies in particular to the Gaussian measure.  
\end{remark}

\begin{remark}
    For $\k=0$, the assumption $(\frac{C_{\k}}{S_{\k}}\phi')'\geq 0$ reduces to $$r\phi''-\phi'\geq 0.$$
    If $\phi^{'''}\geq 0$, we have $(r\phi''-\phi')'=\phi^{'''}\geq 0$, and since $r\phi''(r)-\phi'(r)\geq (r\phi''-\phi')(0)=0$, the condition is satisfied.  Thus many choices of $\phi$
 satisfy the conditions of Lemma \ref{lm3.3},  for example 
$\phi(r)=r^{k}$ with $k\geq 2.$
\end{remark}

\begin{remark}
    For $\k=1$, the assumption $(\frac{C_{\k}}{S_{\k}}\phi')'\geq 0$ reduces to $$\tan (r)\phi''(r)-\phi'(r)\geq 0.$$ Since  $\phi(r(x))$ is smooth and $\phi''(r)\ge 0$, we have $\phi'(0)=0$ and $\phi'(r)\geq 0$. Hence for $r>\frac{\pi}{2}$, the assumption $(\frac{C_{\k}}{S_{\k}}\phi')'\geq 0$ is not satisfied.
\end{remark}

Let $T_1(r)$ be the first eigfenfunction of  \eqref{3.2}, namely
\begin{align}\label{3.16}
T_1'' + \left((n-1)\frac{C_\kappa}{S_\kappa} - \phi'\right)T_1' + \big(\mu_1(B_R)- (n-1)S_\kappa^{-2}\big)T_1 = 0,\quad r\in (0, R)
\end{align}
with $T_1(0)=0$, $T_1'(R)=0$, and $T_1(r)>0$ for $r\in (0, R)$.
Define the quantities:
\begin{equation}\label{3.17}
\begin{aligned}
A_R=&\int_0^R T_{1}(r)^2e^{-\phi(r)}S_\k(r)^{n-1}\, dr,\\
Q_R=&\int_0^R T_{1}'(r)^2e^{-\phi(r)}S_\k(r)^{n-1}\, dr\\
H_R=&\int_0^R \frac{T_{1}(r)^2}{S_{\k}(r)^2}e^{-\phi(r)}S_\k(r)^{n-1}\, dr.
\end{aligned}
\end{equation}
Multiplying \eqref{3.16} by $T_1$ and integrating by parts with respect with respect to the radial measure $e^{-\phi(r)}S_\k(r)dr$ over $(0, R)$, we obtain the following energy identity
\begin{equation}\label{3.18}
Q_R+(n-1)H_R=\mu_{1}(B_R) A_R.
\end{equation}

\section{Proof of Theorem \ref{thm1.1}}\label{sect4}
In this section, we prove Theorem \ref{thm1.1}. Let $\Omega \in \mathcal{E}_\k$ be a connected origin-symmetric Lipschitz domain with weighted volume bounded by that of the ambient space (or by $|B_{\pi/2}|_{\gamma_\k}$ when $\k=1$), and let $B_R$ is geodesic ball of radius $R$ centered at $o$ such that
$$|\Omega|_{\gamma_\k}=|B_R|_{\gamma_\k}.$$
In the following, we interpret the integral over $\Omega$ as the integral over the tangent plane $T_oM_\k$. 

\begin{proof}
Let $(r, \theta)$ denote the polar coordinates centered at $o$, and let $(x_1, \cdots, x_n$)  be  normal coordinates on the tangent plane $T_o M_\k$.
The restrictions of the linear coordinate functions $\psi_i$ on $\mathbb{S}^{n-1}$ are given by
\begin{align*}
\psi_i(\theta)=\frac{x_i}{|x|},
\end{align*}
and satisfy
\begin{align}\label{4.1}
    \sum_{i=1}^n \psi_i^2(\theta)=1.
\end{align}
Let  $dx$ denote the volume element on  the tangent plane. Then the  weighted volume element on 
$M_\k$ is given by 
$$d\gamma_{\k}=(S_{\k}/r)^{n-1}e^{-\phi(r)}dx=S_{\k}(r)^{n-1}e^{-\phi(r)}drd\theta,$$
where $d\theta$  denotes the volume element on $\S^{n-1}$.

Let $T_{1}(r)$ be the first eigenfunction of  \eqref{3.2}, and 
define $G(r) : [0,\infty) \to [0,\infty)$ by
\begin{align*}
G(r) =
\begin{cases}
T_{1}(r), & r < R, \\
T_{1}(R), & r \geq R.
\end{cases}
\end{align*}
For $1\le i\le n$, define the trial functions
\begin{align*}
    v_i(x) = G(r(x))\psi_i(\theta).
\end{align*}
Since $\Omega$ is symmetric about $o$, we have 
\begin{align*}
    \int_{\Omega}v_{i}(x)d\gamma_{\k}=0, \quad i=1,2,\cdots, n.
\end{align*}
The functions $v_1,\ldots,v_n$ are linearly independent. Indeed, suppose $\alpha\in\R^n$ and
\begin{align*}
v_\alpha:=\sum_{i=1}^{N}\alpha_i v_i=0
\quad\text{a.e. in }\Omega.
\end{align*}
By continuity, $v_\alpha=0$ throughout $\Omega$, i.e.,
\begin{align*}
v_\alpha(x)=G(|x|)\frac{\alpha\cdot x}{|x|},
\end{align*}
away from the origin. 
Since $G(r)>0$ for $r>0$, it follows that  $\alpha\cdot x=0$ a.e. in $\Omega$,  hence $\alpha=0$. Thus the functions $v_i$ are linearly independent.

For $1\leq i, j\leq n$, define the matrices $J=(J_{ij})$ and $K=(K_{ij})$ by
\begin{align*}
            J_{ij}=\int_\Omega v_iv_j\,d\gamma_{\k}=\int_\Omega G(r)^2\psi_i\psi_j\,d\gamma_{\k},
\end{align*}
and
\begin{align*}
        K_{ij}=\int_\Omega \langle \nabla v_i,\nabla v_j\rangle_{M_\k}\,d\gamma_{\k},
\end{align*}
where   $\langle \cdot, \cdot\rangle_{M_{\k}}$  denotes the inner product with respect to the metric of $M_\k$. Clearly $J$ is an $L^2$-Gram matrix, and $J\succeq 0$.
Now, we  verify that the matrix $K$ is positive definite. For any $\xi\in \mathbb{R}^{n}$,
\begin{align*}
    \xi^{T}K\xi=\int_{\Omega}\left\langle \nabla (\sum_{i=1}^{n}\xi_{i}v_{i}),\nabla(\sum_{j=1}^{n}\xi_{j}v_{j})\right\rangle_{M_\k} d\gamma_{\k}\geq 0,
\end{align*}
If equality holds, then 
 $\sum_{i=1}^{n}\xi_{i}v_{i}(x)$
 has vanishing  weak gradient and  is therefore constant on the connected domain $\Omega$. 
Since
 \begin{align*}     \int_{\Omega}\sum_{i=1}^{n}\xi_{i}v_{i}d\gamma_{\k}=\sum_{i=1}^{n}\xi_{i}\int_{\Omega}v_{i}d\gamma_{\k}=0,
 \end{align*} 
we conclude that $\sum_{i=1}^{n}\xi_{i}v_{i}(x)=0$. By linear independence of $\{v_{1},\cdots, v_{n}\}$, we conclude $\xi_{i}=0$ for all $i$. Thus $K$ is  positive definite.
Applying Lemma ~\ref{lm2.2} yields 
\begin{align}\label{4.2}
    \sum_{i=1}^{n}\frac{1}{\mu_{i}(\Omega)}\geq \tr(K^{-1}J).
\end{align}
We now compute $K$. Since
$$
\nabla v_i=G'(r)\psi_i(\theta)\nabla r+G(r)\nabla \psi_i(\theta)
$$
and $\langle \nabla r,\nabla\psi_i\rangle_{M_\k} =\p \psi_i/\p r=0$, we have
\begin{align*}
    \left\langle \nabla v_{i},\nabla v_{j}\right\rangle_{M_\k}
    =(G')^{2}\psi_i\psi_j+G^{2}\langle \nabla \psi_i,\nabla \psi_j\rangle_{M_\k}.
\end{align*}
A direct computation in $\R^n$ gives
\begin{align*}
    \langle \nabla \psi_i,\nabla \psi_j\rangle_{{\R^{n}}}
    =\sum_{m=1}^{n}\frac{\partial \psi_i}{\partial x_{m}}\frac{\partial \psi_j}{\partial x_m}
    =\sum_{m=1}^{n}(\frac{\delta_{im}}{r}-\frac{x_{i}x_{m}}{r^{3}})(\frac{\delta_{jm}}{r}-\frac{x_{j}x_{m}}{r^{3}})
    =\frac{\delta_{ij}-\psi_i\psi_j}{r^{2}}.
\end{align*}
Since
$$
\langle \nabla \psi_i,\nabla \psi_j\rangle_{M_\k}=\frac{1}{S_\k(r)^2}\langle \nabla \psi_i,\nabla \psi_j\rangle_{{\S^{n-1}}}=\frac{r^2}{S_\k(r)^2}\langle \nabla \psi_i,\nabla \psi_j\rangle_{{\R^{n}}},
$$
we obtain
\begin{align*}
\langle \nabla \psi_i,\nabla \psi_j\rangle_{M_\k}=\frac{1}{S^{2}_{\k}(r)}(\delta_{ij}-\psi_i\psi_j).
\end{align*}
Therefore
\begin{align*}
    \left\langle \nabla v_{i},\nabla v_{j}\right\rangle_{M_\k}=G'(r)^{2}\psi_i\psi_j+\frac{G(r)^2}{S_{\k}(r)^2}(\delta_{ij}-\psi_i\psi_j).
\end{align*}
Consequently, the matrix $K$ can be expressed as
\begin{align}\label{4.3}
K_{ij}&=\int_{\Omega}G'(r)^{2}\psi_i\psi_jd\gamma_{\k}+\int_{\Omega}\frac{G(r)^{2}}{S_{\k}(r)^{2}}(\delta_{ij}-\psi_i\psi_j)d\gamma_{\k}.
\end{align}
For each $\theta\in\Sph^{n-1}$, define the radial slice
\begin{align*}
        E_\theta:=\{r>0:(r, \theta)\in\Omega\}
\end{align*}
and the corresponding radial weighted measure
\begin{align*}
        Y(\theta)=\int_{E_\theta}S_{\k}(r)^{n-1}e^{-\phi(r)}\,dr,
        \qquad
        Y_R=\int_0^R S_{\k}(r)^{n-1}e^{-\phi(r)}\,dr.
\end{align*}
Using polar coordinates,  the weighted volume of 
$\Omega$ is
\begin{align*}
        |\Omega|_{\gamma_{\k}}
        =\int_{\Sph^{n-1}}\int_{E_\theta} e^{-\phi(r)}S_{\k}(r)^{n-1}\,dr\,d\theta
        =\int_{\Sph^{n-1}}Y(\theta)\,d\theta.
\end{align*}
Similarly,
\begin{align*}
        |B_R|_{\gamma_{\k}}
        =
        \int_{\Sph^{n-1}}Y_R\,d\theta.
\end{align*}
Since $|B_R|_{\gamma_{\k}}=|\Omega|_{\gamma_{\k}}$, it follows that
\begin{equation}\label{4.4}
        \int_{\Sph^{n-1}}(Y(\theta)-Y_R)\,d\theta=0.
\end{equation}
Let 
$$
\psi(\theta)=(\psi_1(\theta), \psi_2(\theta), \cdots, \psi_n(\theta))^\top
$$
and 
define 
\begin{align*}
        Z:=\int_{\Sph^{n-1}}(Y(\theta)-Y_R)\psi \psi^\top \,d\theta.
\end{align*}
By  \eqref{4.1}, we have $\tr (Z)=0$. 
Let
\begin{align*}
        V(r)=\int_0^r S_{\k}(t)^{n-1}e^{-\phi(t)}\,dt,
\end{align*}
and let  $\tau$ denote its inverse, namely $V(\tau(y))=y$. Define
\begin{align*}
\mathcal A(y)=\int_0^{\tau(y)}G(r)^2S_{\k}(r)^{n-1}e^{-\phi(r)}\,dr,
\end{align*}
and
\begin{align*}
\mathcal H(y)=\int_0^{\tau(y)}\frac{G(r)^2}{S_{\k}(r)^2}S_{\k}(r)^{n-1}e^{-\phi(r)}\,dr.
\end{align*}
Direct differentiation yields, for $y>0$
\begin{align*}
\mathcal A'(y)=G(\tau(y))^2,
\qquad
\mathcal H'(y)=\frac{G(\tau(y))^2}{S_{\k}(\tau(y))^2}.
\end{align*}
By Lemma~\ref{lm3.2} and Lemma ~\ref{lm3.3}, $\mathcal A'$ is non-decreasing and
$\mathcal H'$ is strictly decreasing. Hence $\mathcal A$ is convex and
$\mathcal H$ is strictly concave. Since $Y_R=V(R)$ and $\tau(Y_R)=R$, we obtain
\begin{equation}\label{4.5}
\mathcal A(y)\geq \mathcal A(Y_R) +\mathcal A'(Y_R) (y-Y_R) = A_R+G(R)^2(y-Y_R),
\end{equation}
and 
\begin{equation}\label{4.6}
\mathcal H(y)\leq \mathcal H(Y_R)+\mathcal H'(Y_R) (y-Y_R) = H_R+h_R(y-Y_R),
\end{equation}
where $A_R$ and $H_R$ are defined in \eqref{3.17}, and
\begin{equation*}
h_R=\frac{G(R)^2}{S_{\k}(R)^2}.
\end{equation*}
By Lemma ~\ref{lm3.2} and Lemma ~\ref{lm3.3}, $G(r)$ is non-decreasing  and $\frac{G(r)^{2}}{S_{\k}(r)^{2}}$ is strictly decreasing for $0<r<R$. For $\k=1$, since $\Omega$ is contained in an open hemisphere, we have $\frac{G(r)^{2}}{S_{\k}(r)^{2}}$ is strictly decreasing for $r\in E_{\theta}$ for all $\theta\in \mathbb{S}^{n-1}$.
By Lemma~\ref{lm2.3}, the initial interval of the same
$S_{\k}^{n-1}(r)e^{-\phi(r)}dr$-measure minimizes the integral of the non-decreasing
function $G^2$ and maximizes the integral of the strictly decreasing
function $G(r)^2/S_{\k}(r)^2$. Applying this principle to the radial slices $E_\theta$ and combining with \eqref{4.5} and \eqref{4.6}, we obtain, for every $\theta\in\S^{n-1}$,
\begin{equation}\label{4.7}
        \int_{E_\theta}G(r)^2 S_{\k}(r)^{n-1}e^{-\phi(r)}\,dr
        \geq A_R+G(R)^2(Y(\theta)-Y_R),
\end{equation}
and
\begin{equation}\label{4.8}
        \int_{E_\theta}\frac{G(r)^2}{S_{\k}(r)^2}S_{\k}(r)^{n-1}e^{-\phi(r)}\,dr
        \leq H_R+h_R(Y(\theta)-Y_R).
\end{equation}
Moreover, since $G'=0$ on $[R,\infty)$, we have
\begin{equation}\label{4.9}
        \int_{E_\theta}G'(r)^2S_{\k}(r)^2e^{-\phi(r)}\,dr
        \leq Q_R,
\end{equation}
where $Q_R$ is defined in \eqref{3.17}.

 Using polar coordinates,  the matrix  $J$ can be written as
\begin{align*}
J=\int_{\S^{n-1}}\left(\int_{E_\theta}G(r)^2S_{\k}(r)^{n-1}e^{-\phi(r)}\,dr\right)\psi\psi^\top\,d\theta.
\end{align*}
 For any  $v\in\mathbb R^n$, applying  \eqref{4.7} yields
\begin{align*}
\begin{aligned}
v^\top  Jv
        &=
        \int_{\Sph^{n-1}}
        \left(
        \int_{E_\theta}G(r)^2S_{\k}(r)^{n-1}e^{-\phi(r)}\,dr
        \right)
        (v\cdot\psi)^2\,d\theta \\
        &\geq
        A_R\int_{\Sph^{n-1}}(v\cdot\psi)^2\,d\theta
        +
        G(R)^2
        \int_{\Sph^{n-1}}(Y(\theta)-Y_R)(v\cdot\psi)^2\,d\theta.
\end{aligned}
\end{align*}
Equivalently,
\begin{align*}
        J\succeq
        A_R\int_{\Sph^{n-1}}\psi\psi^\top \,d\theta
        +
        G(R)^2
        \int_{\Sph^{n-1}}(Y(\theta)-Y_R)\psi\psi^\top \,d\theta .
\end{align*}
Using the identity
\begin{align}\label{4.10}
    \int_{\S^{n-1}} \psi_i\psi_j\, d\theta=\frac{\omega_{n-1}}{n} \delta_{ij},
\end{align}
where $\omega_{n-1}=|\S^{n-1}|$, we obtain
\begin{equation}\label{4.11}
        J\succeq aI+cZ,
\end{equation}
with 
$$a=\frac{\omega_{n-1}}{n} A_R, \qquad c=G(R)^2.$$ 
We next estimate $K$ from above. From \eqref{4.3}, again in
polar coordinates,
\begin{align}\label{4.12}
K=&
\int_{\S^{n-1}} \left(\int_{E_\theta}G'(r)^2S_{\k}(r)^{n-1}e^{-\phi(r)}\,dr\right)\psi\psi^\top d\theta\nonumber\\
    &+\int_{\S^{n-1}} \left(\int_{E_\theta}\frac{G(r)^2}{S_{\k}(r)^2}S_{\k}(r)^{n-1}e^{-\phi(r  )}\,dr\right) (I-\psi\psi^\top )
       d\theta.
\end{align}
Applying the bounds \eqref{4.8} and
\eqref{4.9} to \eqref{4.12}, we get
\begin{align*}
K
\preceq
\int_{\Sph^{n-1}}
        \left[
        Q_R\psi\psi^\top
        +
        \bigl(H_R+h_R(Y(\theta)-Y_R)\bigr)(I-\psi\psi^\top )
        \right]d\theta,
\end{align*}
where we used the fact that both $\psi\psi^\top$ and  $I- \psi\psi^\top$ are
positive semidefinite.
Expanding the right-hand side and using the identity \eqref{4.10}, we obtain
\begin{align*}
\begin{aligned}
K
\preceq&
Q_R\int_{\Sph^{n-1}}\psi\psi^\top \,d\theta
+
H_R\int_{\Sph^{n-1}}(I-\psi\psi^\top )\,d\theta\\
&+
h_R\int_{\Sph^{n-1}}(Y(\theta)-Y_R)(I-\psi\psi^\top)\,d\theta \\
=&\frac{\omega_{n-1}}{n}\bigl(Q_R+(n-1)H_R\bigr)I-h_RZ.
\end{aligned}
\end{align*}
By the energy identity \eqref{3.18}, this becomes
\begin{equation}\label{4.13}
        K\preceq \mu_{1}(B_R) aI-dZ,
\end{equation}
where
\begin{equation*}
        d=h_R=\frac{G(R)^2}{R^2}>0.
\end{equation*}

Recall that  $K\succ 0$, $J\succeq0$, $\tr Z=0$, and
$a,c,d$ are all positive.  Thus the hypotheses of Lemma~\ref{lm2.1} are satisfied by
\eqref{4.11} and \eqref{4.13}. Applying Lemma~\ref{lm2.1}
gives
\begin{align*}
        \tr(K^{-1}J)\geq \frac{n}{\mu_{1}(B_R)}.
\end{align*}
Combining this with \eqref{4.2},we obtain the desired inequality
\begin{align}\label{4.14}
    \sum_{i=1}^{n}\frac{1}{\mu_{i}(\Omega)}\geq \frac{n}{\mu_{1}(B_R)}.
\end{align}

It remains to discuss the equality case. From Lemma  \ref{lm2.1}, equality in  \eqref{4.14} implies
$$
Z=0, \quad J=aI, \quad K=\mu_{1}(B_R)aI.
$$
In particular, 
$\tr J=na$, which gives 
\begin{align}\label{4.15}
    \int_{\Omega}G(r)^{2}d\gamma_{\k}=\int_{B_R}G(r)^{2}d\gamma_{\k}.
\end{align}
Using \eqref{4.15} and the definition of $G$, we have 
\begin{align*}
\int_{B_R\setminus\Omega}G(r)^{2}d\gamma_{\k}=G(R)^{2}\int_{\Omega\setminus B_R}d\gamma_{\k}=G(R)^{2}\int_{B_R\setminus \Omega}d\gamma_{\k}.
\end{align*}
Since $G(r)$ is strictly increasing on $(0,R)$, it follows
$$
0=\int_{B_R\setminus \Omega}G(R)^{2}-G(r)^{2}d\gamma_{\k}\geq 0.
$$
Moreover $|B_R\setminus \Omega|_{\gamma_{\k}}=0$ and $|\Omega\setminus B_R|_{\gamma_{\k}}=0$. 
Since the density of $d\gamma_{\k}=e^{-\phi(r)}d\V_k$ is positive, we obtain
\begin{align*}
    |\Omega\setminus B_R|+   |B_R\setminus \Omega|=0.
\end{align*}
As $\Omega$ is a Lipschitz domain, it follows (as in \cite{Fr26}),  $\Omega=B_R$ up to a set of measure zero, and hence $\Omega$ coincides with $B_R$. This completes the proof of Theorem \ref{thm1.1}.
\end{proof}

 \bibliographystyle{plain}
	\bibliography{ref}
\end{document}